\documentclass[11pt, reqno]{amsart}
\usepackage{amsmath, amsthm, amscd, amsfonts, amssymb, mathrsfs, graphicx, color}
\usepackage[bookmarksnumbered, colorlinks, plainpages]{hyperref}
\hypersetup{colorlinks=true,linkcolor=red, anchorcolor=green, citecolor=cyan, urlcolor=red, filecolor=magenta, pdftoolbar=true}

\theoremstyle{plain}
\newtheorem{theorem}{Theorem}[section]
\newtheorem{corollary}[theorem]{Corollary}
\newtheorem{lemma}[theorem]{Lemma}
\newtheorem{proposition}[theorem]{Proposition}

\theoremstyle{definition}
\newtheorem{definition}{Definition}[section]

\theoremstyle{remark}
\newtheorem{remark}{Remark}[section]
\newtheorem{example}{Example}[section]

\numberwithin{equation}{section}

\numberwithin{figure}{section} 
\begin{document}
\title[Functional Weighted Hermite-Hadamard Inequality]
{On a Weighted Hermite-Hadamard Inequality Involving Convex Functional Arguments}
\author[M. Ra\"{\i}ssouli, M. Chergui, L. Tarik]{Mustapha Ra\"{\i}ssouli$^{1}$, Mohamed Chergui$^2$ and Lahcen Tarik$^{3}$}
\address{$^{1}$ Department of Mathematics, Science Faculty, Moulay Ismail University, Meknes, Morocco.}
\address{$^{2}$ Department of Mathematics, CRMEF-RSK, EREAM Team, LaREAMI-Lab, Kenitra, Morocco.}
\address{$^{3}$ LAGA-Laboratory, Science Faculty,  Ibn Tofail University, Kenitra, Morocco.}

\email{\textcolor[rgb]{0.00,0.00,0.84}{raissouli.mustapha@gmail.com}}
\email{\textcolor[rgb]{0.00,0.00,0.84}{chergui\_m@yahoo.fr}}
\email{\textcolor[rgb]{0.00,0.00,0.84}{lahcen.tarik@uit.ac.ma}}


\begin{abstract}
In this paper, we are interested in investigating a weighted variant of Hermite-Hadamard type inequalities involving convex functionals. The approach undertaken makes it possible to refine and reverse certain inequalities already known in the literature. It also allows us to provide new weighted convex functional means and establish some related properties with respect to some standard means.
\end{abstract}

\maketitle

\section{Introduction}\label{sec1}

Let $C$ be a nonempty interval of ${\mathbb R}$ and $f:C\rightarrow{\mathbb R}$ be a convex function. The Hermite-Hadamard inequalities
\begin{equation}\label{1h}
f\Big(\dfrac{a+b}{2}\Big)\leq\frac{1}{b-a}\int_a^bf(t)dt\leq\dfrac{f(a)+f(b)}{2}
\end{equation}
hold for any $a,b\in C,\; a\neq b$. If $f:C\rightarrow{\mathbb R}$ is concave then \eqref{1h} are reversed. By virtue of their usefulness in mathematical analysis and in other areas as information theory, inequalities \eqref{1h} attracted many researchers to investigate extensions, refinements and reverses as one can see in \cite{BT,D2,D3,D4,M,NIC} for instance, and the related references cited therein.

For operator maps, inequalities \eqref{1h} have been extended in the following manner, see \cite{D2,D3,D4} for instance. An operator convex function $\varphi$ on a nonempty interval $J$ of $\mathbb{R}$ is such that the inequality
\begin{equation}\label{0h}
\varphi\big((1-\lambda)A+\lambda B\big)\leq(1-\lambda)\varphi(A)+\lambda\varphi(B)
\end{equation}
holds for any $\lambda\in[0,1]$ and $A$ and $B$ are two self-adjoint operators acting on a complex Hilbert space $H$, with $Sp(A), Sp(B) \subset J$. With this, the operator version of \eqref{1h} reads as follows
\begin{equation}\label{1}
\varphi\Big(\dfrac{A+B}{2}\Big)\leq\int_0^1\varphi\Big((1-t)A+tB\Big)dt\leq\dfrac{\varphi(A)+\varphi(B)}{2},
\end{equation}
with reversed inequalities if \eqref{0h} is reversed, i.e. $\varphi$ is operator concave on $J$. Here, $f(A)$ is defined by the techniques of functional calculus as usual, and $\leq$ refers to the L\"{o}wner partial order defined as follows: $A\le B$ if and only if $A,B$ are self-adjoint and $B-A$ is positive.

In \cite{DR}, the authors provided an extension of \eqref{1} for pointwise convex maps involving functional arguments as explained in what follows. Let $\mathcal{C}$ be a nonempty subset of $\widetilde{\mathbb{R}}^H$, where $\widetilde{\mathbb{R}}^H$ denotes the (extended) space of functionals defined from $H$ into $\widetilde{\mathbb{R}}:= \mathbb{R} \cup \{+\infty\}$. A functional map $\Phi:\mathcal{C} \rightarrow \widetilde{\mathbb{R}}^H$ is said to be pointwise convex (resp. concave) if for all $f, g \in \mathcal{C}$ and all real number $\lambda\in[0; 1]$, we have
\begin{equation}\label{convpointwise}
\Phi\Big((1-\lambda)f+\lambda g\Big)\leq(\geq) (1-\lambda)\Phi(f)+\lambda \Phi(g).
\end{equation}
For these class of functional maps, the following inequalities
\begin{equation}\label{2}
\Phi\Big(\dfrac{f+g}{2}\Big)\leq(\geq)\int_0^1\Phi\Big((1-t)f+tg\Big)dt\leq(\geq)\dfrac{\Phi(f)+\Phi(g)}{2},
\end{equation}
hold for any $f,g\in\mathcal{C}$. Here, $\le$ refers to the pointwise order defined by $f \le g$ if and only if $f(x)\le g(x)$ for all $x \in H$.

A typical and special case of pointwise convex map is $\Phi(f)=f^*$, where $f^*$ stands for the Fenchel conjugate in convex analysis defined by the following formula
\begin{gather}\label{4}
\forall x^* \in H \quad f^*(x^*):=\sup_{x\in H}\Big\{\Re e<x^*,x> -f(x)\Big\}.
\end{gather}
In this case, \eqref{2} immediately imply that,
\begin{equation}\label{3}
\Big(\dfrac{f+g}{2}\Big)^*\leq\int_0^1\Big((1-t)f+tg\Big)^*dt\leq\dfrac{f^*+ g^*}{2}.
\end{equation}
For more examples of pointwise convex/concave maps, we can consult \cite{DR}.

Now, let $A$ be a (self-adjoint) positive operator defined from the Hilbert space $(H,\langle.,.\rangle)$ into itself. An important and useful example of convex functional is $f_A$ defined by
\begin{gather*}
\forall x \in H\quad f_A(x)= \frac{1}{2}\langle Ax,x\rangle,
\end{gather*}
which is called the quadratic convex function associated to $A$. If moreover $A$ is invertible then we have, \cite{RC2001}
\begin{gather*}
\forall x \in H\quad f_A^*(x)= \frac{1}{2}\langle A^{-1}x,x\rangle,
\end{gather*}
which we briefly write $f^*_A= f_{A^{-1}}$. Taking $f=f_A$ and $g=f_B$, with $A$ and $B$ both positive invertible acting on $H$, we then deduce that the operator version of \eqref{3} is
\begin{equation}\label{5}
\Big(\dfrac{A+B}{2}\Big)^{-1}\leq\int_0^1\Big((1-t)A+tB\Big)^{-1}dt\leq\dfrac{A^{-1}+ B^{-1}}{2}.
\end{equation}
It is worth mentioning that \eqref{5} can be also deduced from \eqref{1} when choosing $\varphi(x)=1/x$ which is operator convex on $(0,\infty)$.

For the sake of clearness for the reader, it is important to mention that one of the most effective tools invested in developing Hermite-Hadamard inequalities is the Jensen inequality, which reads in its general form as follows \cite[p.202]{GenJens}. Let $\Omega$ be a $\mu$-measurable set such that $\mu\big(\Omega\big)>0$. Let $E$ be a nonempty convex subset of $\mathbb{R}$ and $f : E \longrightarrow \mathbb{R }$ be a convex function. Let $\phi \in  L^1(\Omega)$ be such that $\phi(x)\in E$  almost
everywhere and $f\circ \phi \in L^1(\Omega)$. Then, we have the integral Jensen inequality
\begin{gather}\label{Jens}
 f\left( \dfrac{1}{\mu\big(\Omega\big)} \int_\Omega\,\phi(x)\,d\mu(x)\right) \le  \dfrac{1}{\mu\big(\Omega\big)} \int_\Omega\,f\big(\phi(x)\big)\,d\mu(x).
\end{gather}

The fundamental target of this work is to investigate another variant of the previous functional Hermite-Hadamard inequalities. This makes possible the refinement of \eqref{2} and allows the improvement of some inequalities involving convex functional arguments.

The rest of this paper will be organized as follows. In Section \ref{sect2} we collect some basic notions about functional means that will be needed in the sequel. Section \ref{sect3} is devoted to establishing weighted variant of Hermite-Hadamard inequalities involving convex functional arguments. In Section \ref{sect4}, we present some new weighted functional means as well as some of their related properties.

\section{Operator and functional means}\label{sect2}

For over the last two centuries, the mean-theory in scalar variables has extensive several developments and various applications, \cite{DP}. This theory was then extended for operator variables \cite{Burqan et al(2023),Choi and Kim(2023),kubo and ando,Udagawa(2017)} and later for functional arguments via the Fenchel conjugation in convex analysis, see \cite{RC2001,RaiBou} for instance. We omit here to recall the notion of scalar/operator means which are out of the aim of this paper and we just refer the interested readers to the previous mentioned references. For means involving functional arguments, we will however recall in the following some of them that will be needed throughout this manuscript. Let $\Gamma_0(H)$ be the cone of convex lower semi-continuous functionals defined from $H$ into $\tilde{\mathbb R}$ which are not identically equal $+\infty$ and dot not take the value $-\infty$, see \cite{A,B,ET} for instance.

For $f,g\in\Gamma_0(H)$ and $\lambda\in(0,1)$, the following functional expressions
\begin{equation}\label{ArHaGe}
f\nabla_\lambda g:=(1-\lambda)f+\lambda g,\;\; f!_\lambda g:=\Big((1-\lambda)f^*+\lambda g^*\Big)^*, \;
f\sharp_\lambda g:=\int_{0}^{1} f!_t g\,d\mu_\lambda(t),
\end{equation}
where
$$d\mu_\lambda(t):=\frac{\sin(\pi \lambda)}{\pi}\frac{t^{\lambda-1}}{^(1-t)^\lambda}\;dt,$$
are known, by analogy with the scalar/opertor version, as the $\lambda$-weighted arithmetic mean, the $\lambda$-weighted harmonic mean and the $\lambda$-weighted geometric mean of $f$ and $g$, respectively. They satisfy
\begin{equation}\label{IAGH}
f!_\lambda g \le f\sharp_\lambda g  \le f\nabla_\lambda g
\end{equation}
analogously to the scalar/operator version. For $\lambda = 1/2$, they are simply denoted by $f\nabla g, f!g$ and $f\sharp g$, respectively, and we extend them to the whole interval $[0,1]$ by setting:
$$f\nabla_0 g = f!_0 g= f\sharp_0 g =f,\;\; f\nabla_1 g = f!_1 g= f\sharp_1 g =g.$$
These extensions are taken as definition and can not be in fact deduced from \eqref{ArHaGe}, by virtue of the usual convention $0.(+\infty):=+\infty$ adopted in convex analysis, see \cite{A,B,ET} for instance.

In \cite{RaiJipam,RaiFur} two other functional means have been introduced, namely
\begin{equation}\label{logmean1}
\mathbf{L}(f,g):= \left(\int_{0}^{1}\, (f\nabla_t g)^* \,dt\right)^*\;\; \text{ and }\;\; L(f,g):= \int_{0}^{1}\, f\sharp_t g \,dt,
\end{equation}
which are both called logarithmic functional means of $f$ and $g$. It is worth mentioning that $\mathbf{L}(f,g)$ and $L(f,g)$ coincide in the scalar/operator version, i.e. when $f=f_A$ and $g=f_B$, for some positive invertible operators $A$ and $B$, and are reduced to the logarithmic operator mean of $A$ and $B$ already introduced in \cite{RaiJipam} for example. However, the question to know if $\mathbf{L}(f,g)$ and $L(f,g)$ coincide, for any $f, g\in\Gamma_0(H)$, is still an open problem.

\section{Main results}\label{sect3}

Let $\lambda\in(0,1)$. For any $t\in[0,1]$ we set
\begin{equation}\label{22}
d\nu_\lambda(t)=\Big((1-\lambda)(1-t)^{\frac{1-2\lambda}{\lambda}}+\lambda t^{\frac{2\lambda-1}{1-\lambda}}\Big)dt.
\end{equation}

Then, we have the following lemma.

\begin{lemma}
(i) $\nu_\lambda$ is a probability measure on $[0,1]$.\\
(ii) For $0\le a< b\le 1$ we have
  \begin{multline}\label{integ}
\int_{a}^{b} td\nu_\lambda(t) = \xi_{a,b}^\lambda:= \lambda (1-\lambda)\Big((1-b)^{\frac{1}{\lambda}}+b^{\frac{1}{1-\lambda}}-(1-a)^{\frac{1}{\lambda}}-a^{\frac{1}{1-\lambda}}\Big)\\
-\lambda\Big((1-b)^{\frac{1-\lambda}{\lambda}}-(1-a)^{\frac{1-\lambda}{\lambda}}\Big),
\end{multline}
\begin{equation}\label{integr}
\nu_\lambda\big([a,b]\big):=\int_{a}^{b} d\nu_\lambda(t) = \chi_{a,b}^\lambda 
:= \lambda \Big((1-a)^{\frac{1-\lambda}{\lambda}}-(1-b)^{\frac{1-\lambda}{\lambda}}\Big)+
(1-\lambda)\Big(b^{\frac{\lambda}{1-\lambda}}-a^{\frac{\lambda}{1-\lambda}}\Big).
\end{equation}
\end{lemma}
\begin{proof}
Simple calculations from Real Analysis area lead to the desired results. We omit the details here.
\end{proof}

Our first main result reads as follows.

\begin{theorem}\label{thf}
Let $\mathcal{C}$ be a nonempty convex subset of $\widetilde{\mathbb{R}}^H$ and $\Phi:\mathcal{C}\rightarrow \widetilde{\mathbb{R}}^H$ be a pointwise convex map. For any $\lambda\in(0,1), 0\le a< b\le1$ and $f, g\in \mathcal{C}$, the following inequalities hold
\begin{multline}\label{20}
\chi_{a,b}^\lambda\,\Phi\left(\frac{1}{\chi_{a,b}^\lambda}\Big((1-\xi_{a,b}^\lambda)f+\xi_{a,b}^\lambda\;g\,
\Big)\right)\leq\int_a^b\Phi\Big((1-t)f+tg\Big)d\nu_\lambda(t)\\
\leq (1-\xi_{a,b}^\lambda)\Phi(f)+\xi_{a,b}^\lambda \Phi(g).
\end{multline}
If $\Phi$ is pointwise concave then the inequalities \eqref{20} are reversed.
\end{theorem}
\begin{proof}
By the use of \eqref{Jens} and the pointwise convexity of $\Phi$, we can write
\begin{multline}\label{23}
\Phi\left(\frac{1}{\nu_\lambda\big([a,b]\big)}\int_a^b\big((1-t)f+tg\big)\,d\nu_\lambda(t)\right)\leq \frac{1}{\nu_\lambda\big([a,b]\big)}\int_a^b\,\Phi\Big((1-t)f+tg\Big)d\nu_\lambda(t)\\
\leq \frac{1}{\nu_\lambda\big([a,b]\big)}\int_a^b\Big((1-t)\Phi(f)+t\Phi(g)\Big)d\nu_\lambda(t).
\end{multline}
Using \eqref{integ} and \eqref{integr}, we have
$$\int_a^b\big((1-t)f+tg\big)d\nu_\lambda(t)=(1-\xi_{a,b}^\lambda)f+\xi_{a,b}^\lambda g \text{ and }  \nu_\lambda\big([a,b]\big)= \chi_{a,b}^\lambda,$$
hence the inequalities \eqref{20}.
\end{proof}

From Theorem \ref{thf}, we can deduce the following important result.

\begin{corollary}
Let $\mathcal{C}$ be a nonempty convex subset of $\widetilde{\mathbb{R}}^H$ and $\Phi:\mathcal{C}\rightarrow \widetilde{\mathbb{R}}^H$ be a pointwise convex map. For any $\lambda\in(0,1)$ and $f, g\in \mathcal{C}$, the following inequalities hold
\begin{equation}\label{21}
\Phi\Big((1-\lambda)f+\lambda g\Big)\leq\int_0^1\Phi\Big((1-t)f+tg\Big)d\nu_\lambda(t)\leq(1-\lambda)\Phi(f)+\lambda \Phi(g).
\end{equation}
If $\Phi$ is pointwise concave, the inequalities \eqref{21} are reversed.
\end{corollary}
\begin{proof}
Taking $a=0$ and $b=1$ in \eqref{integ} and \eqref{integr} we find $\chi_{a,b}^\lambda=1$ and $ \xi_{a,b}^\lambda=\lambda$. Substituting these in the inequalities \eqref{20}, we get the desired result.
\end{proof}

Throughout this paper, inequalities \eqref{21} will be called the weighted Hermite-Hadamard functional inequalities, $(WHHFI)$ in brief.

\begin{remark} (i) If $\lambda= 1/2$ then $d\nu_{1/2}(t)=dt$ and consequently \eqref{21} corresponds to \eqref{2}.\\
(ii) By \eqref{22}, we have the following equality
\begin{gather}\label{measurerelation}
d\nu_{1-\lambda}(1-t)=d\nu_{\lambda}(t)\;\; \text{for any}\;\; \lambda\in(0,1),\; t\in[0,1].
\end{gather}
\end{remark}

To state more results, let us define the following notations. For $\Phi:\mathcal{C}\rightarrow \widetilde{\mathbb{R}}^H$, $\lambda\in(0,1)$ and $f,g\in{\mathcal C}$, we put
\begin{equation}\label{24}
{\mathcal M}_\lambda(\Phi;f,g):=\int_0^1\Phi\big(f\nabla_tg\big)d\nu_\lambda(t),
\end{equation}
and
$$\mathbf{J}_\Phi(f,g):=\int_0^1{\mathcal M}_\lambda(\Phi;f,g)\,d\lambda.$$

The following corollary provides an immediate application of inequalities \eqref{21}. It concerns a refinement of \eqref{2}.

\begin{corollary}
Let $\Phi:\mathcal{C}\rightarrow \widetilde{\mathbb{R}}^H$ be a pointwise convex map. For any $\lambda\in(0,1)$ and $f, g\in{\mathcal C}$, the following inequalities hold
\begin{equation}\label{2f1}
\Phi\big(f\nabla g\big)\leq \int_0^1\Phi\big(f\nabla_tg\big)dt\leq \mathbf{J}_\Phi(f,g)\leq \Phi(f)\nabla \Phi(g).
\end{equation}
If $\Phi$ is pointwise concave then \eqref{2f1} are reversed.
\end{corollary}
\begin{proof}
Integrating \eqref{21} with respect to $\lambda\in(0,1)$ and using the left inequality in \eqref{2}, we immediately get \eqref{2f1}.
\end{proof}

An extension of the map $\lambda\mapsto{\mathcal M}_\lambda(\Phi;f,g)$ to the whole interval $[0,1]$ is possible, as emphasized in the following result.

\begin{corollary}
Let $\Phi:\mathcal{C}\rightarrow \widetilde{\mathbb{R}}^H$ be a pointwise convex (resp. concave) map. For any $f, g\in{\mathcal C}$, we have
$$\lim_{\lambda\downarrow0}{\mathcal M}_\lambda(\Phi;f,g))=\Phi(f)\;\; \mbox{and}\;\; \lim_{\lambda\uparrow1}{\mathcal M}_\lambda(\Phi;f,g)=\Phi(g).$$
\end{corollary}
\begin{proof}
The desired results follow from \eqref{21} by using the fact that if $\Phi:\mathcal{C}\rightarrow \widetilde{\mathbb{R}}^H$ is pointwise convex then it is pointwise continuous on $\mathcal{C}$.
\end{proof}

\begin{proposition}
For any $\Phi:\mathcal{C}\rightarrow \widetilde{\mathbb{R}}^H$, $f, g\in{\mathcal C}$ and $\lambda\in[0,1]$, we have
\begin{equation}\label{35}
{\mathcal M}_{1-\lambda}(\Phi;f,g)={\mathcal M}_{\lambda}(\Phi;g,f).
\end{equation}
\end{proposition}
\begin{proof}
By \eqref{22} and \eqref{24}, we have
$${\mathcal M}_{1-\lambda}(\Phi;f,g)=\int_0^1\Phi\big(f\nabla_t\,g\big)d\nu_{1-\lambda}(t).$$
Performing the change of variable with $t=1-u$, employing \eqref{measurerelation} and \eqref{24}, the proof of the desired relation is achieved.
\end{proof}

To establish further results, the following lemma \cite{RAI0} will be needed. It provides a refinement of the classical Jensen inequality.

\begin{lemma}
Let $\Phi:\mathcal{C}\rightarrow \widetilde{\mathbb{R}}^H$ be a pointwise convex map. The following inequalities
\begin{multline}\label{26}
r(a,b)\Big(\Phi(f)\nabla_a\Phi(g)-\Phi\big(f\nabla_ag\big)\Big)
\leq \Phi(f)\nabla_b\Phi(g)-\Phi\big(f\nabla_bg\big)\\
\leq R(a,b)\Big(\Phi(f)\nabla_a\Phi(g)-\Phi\big(f\nabla_ag\big)\Big),
\end{multline}
hold for any $f, g\in \mathcal{C}$ and $a, b\in(0,1)$ where we set
\begin{equation}\label{27}
r(a,b):=\min\left(\frac{b}{a},\frac{1-b}{1-a}\right),\;\; R(a,b):=\max\left(\frac{b}{a},\frac{1-b}{1-a}\right).
\end{equation}
If $\Phi$ is pointwise concave then \eqref{26} are reversed.
\end{lemma}

We also need the following result.

\begin{lemma}
For any $a,b,\lambda\in(0,1)$, the following relations hold
\begin{equation}\label{28}
\int_0^1r(a,b)\;d\nu_\lambda(b)=\alpha_{a,1-\lambda}+\alpha_{1-a,\lambda}
\end{equation}
\begin{equation}\label{29}
\int_0^1R(a,b)\;d\nu_\lambda(b)=\dfrac{\lambda}{a}+\dfrac{1-\lambda}{1-a}-\alpha_{a,1-\lambda}-\alpha_{1-a,\lambda},
\end{equation}
where, for $x,y>0$, we set
$$\alpha_{x,y}:=y^2\dfrac{1-x^{\frac{1-y}{y}}}{1-x}.$$
\end{lemma}

\begin{proof}
Let us notice that $0\leq b\leq a$ if and only if $\frac{b}{a}\leq\frac{1-b}{1-a}$. Thus we can write
$$\int_0^1r(a,b)\;d\nu_\lambda(b)=\frac{1}{a}\int_0^ab.d\nu_\lambda(b)+\frac{1}{1-a}\int_a^1(1-b).d\nu_\lambda(b).$$
Using \eqref{22} and some simple integral techniques, we obtain \eqref{28}. The relation \eqref{29} can be deduced from \eqref{28} by employing the following relation
$$R(a,b)+r(a,b)=\dfrac{1}{1-a}+\dfrac{1-2a}{a(1-a)}b.$$
The details are straightforward and therefore omitted here.
\end{proof}

We are now in the position to state the following result which provides a refinement and a reverse of \eqref{21}.

\begin{theorem}\label{th2}
Let $\Phi:\mathcal{C}\rightarrow \widetilde{\mathbb{R}}^H$ be a pointwise convex map. For any $a,\lambda\in[0,1]$ and $f, g\in \mathcal{C}$, the following inequalities hold
\begin{multline}\label{2.10}
m(a,\lambda)\Big(\Phi(f)\nabla_a\Phi(g)-\Phi\big(f\nabla_ag\big)\Big)
\leq \Phi(f)\nabla_\lambda \Phi(g)-\int_0^1\Phi\big(f\nabla_tg\big)d\nu_\lambda(t)\\
\leq M(a,\lambda)\Big(\Phi(f)\nabla_a\Phi(g)-\Phi\big(f\nabla_ag\big)\Big),
\end{multline}
where we set
$$m(a,\lambda):=(1-\lambda)^2\dfrac{1-a^{\frac{\lambda}{1-\lambda}}}{1-a}+\lambda^2\dfrac{1-(1-a)^{\frac{1-\lambda}{\lambda}}}{a}
\text{ and }
 M(a,\lambda):=\dfrac{1-\lambda}{1-a}+\dfrac{\lambda}{a}-m(a,\lambda).$$
If $\Phi$ is pointwise concave then \eqref{2.10} are reversed.
\end{theorem}
\begin{proof}
Multiplying all sides of \eqref{26} by $d\nu_\lambda(b)$ and then integrating with respect to $b\in[0,1]$, we obtain the desired inequalities by taking into account the relations \eqref{28} and \eqref{29}. The details are simple and therefore omitted here for the reader.
\end{proof}

Choosing $a=1/2$ in Theorem \ref{th2} and making some algebraic computations, we obtain the following result.

\begin{corollary}\label{cor1f}
Let $\Phi:\mathcal{C}\rightarrow \widetilde{\mathbb{R}}^H$ be pointwise convex, $\lambda\in[0,1]$ and $f, g\in \mathcal{C}$. Then we have
\begin{multline}\label{3.11}
\alpha(\lambda)\Big(\Phi(f)\nabla \Phi(g)-\Phi\big(f\nabla g\big)\Big)
\leq \Phi(f)\nabla_\lambda \Phi(g)-\int_0^1\Phi\big(f\nabla_tg\big)d\nu_\lambda(t)\\
\leq \beta(\lambda)\Big(\Phi(f)\nabla \Phi(g)-\Phi\big(f\nabla g\big)\Big),
\end{multline}
where, $\alpha(\lambda):=m(1/2, \lambda) \text{ and } \beta(\lambda):=M(1/2, \lambda)$.

If $\Phi$ is pointwise concave then \eqref{2.10} are reversed.
\end{corollary}

To point out results involving harmonic means, we recall the following result, \cite{RAI0}.

\begin{lemma}
For any $f,g\in \widetilde{\mathbb{R}}^H$ and $a, b\in(0,1)$, we have the following inequalities
\begin{gather}\label{26f}
r(a,b)\Big(f\nabla_ag-f!_ag\big) \leq f\nabla_bg-f!_bg \leq R(a,b)\Big(f\nabla_ag-f!_ag\Big),
\end{gather}
where $r(a,b)$ and $R(a,b)$ as defined in \eqref{27}.
\end{lemma}

\begin{theorem}\label{th2f}
Let $\Phi:\mathcal{C}\rightarrow \widetilde{\mathbb{R}}^H$ be pointwise convex. For any $a,\lambda\in[0,1]$ and $f, g\in \mathcal{C}$, the following inequalities hold
\begin{multline}\label{2.10ff}
m(a,\lambda)\Big(\Phi(f)\nabla_a\Phi(g)-\Phi(f)!_a\Phi(g)\Big)
\leq\Phi(f)\nabla_\lambda\Phi(g)-\int_0^1\Phi(f)!_t\Phi(g)\,d\nu_\lambda(t)\\
\leq M(a,\lambda)\Big(\Phi(f)\nabla_a\Phi(g)-\Phi(f)!_a\Phi(g)\Big),
\end{multline}
where, $m(a,\lambda)$ and $ M(a,\lambda)$ stand for the real numbers defined in Theorem \ref{th2}.
\end{theorem}
\begin{proof}
Utilizing \eqref{26f} with $\Phi(f)\in\widetilde{\mathbb{R}}^H\in$ and $\Phi(g)\in\widetilde{\mathbb{R}}^H$ instead of $f$ and $g$, respectively, we get
\begin{multline}\label{27f}
r(a,t)\Big(\Phi(f)\nabla_a\Phi(g)-\Phi(f)!_a\Phi(g)\big)
\leq \Phi(f)\nabla_t\Phi(g)-\Phi(f)!_t\Phi(g)\\
\leq R(a,t)\Big(\Phi(f)\nabla_a\Phi(g)-\Phi(f)!_a\Phi(g)\Big),
\end{multline}
Multiplying all sides of \eqref{27f} by $d\nu_\lambda(t)$ and then integrating with respect to $t\in[0,1]$, we obtain the desired inequalities.
\end{proof}

Choosing $a=1/2$ in \eqref{2.10ff} and making some elementary calculations, we get the following corollary.

\begin{corollary}
Let $\Phi:\mathcal{C}\rightarrow \widetilde{\mathbb{R}}^H$ be pointwise convex, $\lambda\in[0,1]$ and $f, g\in \mathcal{C}$. Then
\begin{multline}\label{3.11}
\alpha(\lambda)\Big(\Phi(f)\nabla\Phi(g)-\Phi(f)!\Phi(g)\Big)
\leq\Phi(f)\nabla_\lambda\Phi(g)-\int_0^1\Phi(f)!_t\Phi(g)\,d\nu_\lambda(t)\\
\leq \beta(\lambda)\Big(\Phi(f)\nabla\Phi(g)-\Phi(f)!\Phi(g)\Big).
\end{multline}
where, $\alpha(\lambda)$ and $\beta(\lambda)$ stand for the real numbers defined in Corollary \ref{cor1f}.
\end{corollary}

\section{New weighted functional means}\label{sect4}

In this section, we introduce two new weighted logarithmic means involving convex functionals. As pointed out in \cite{RaiBou}, for fixed $f,g\in \Gamma_0(H)$, the maps $t \longmapsto f\nabla_t g$ and $t \longmapsto f\sharp_t g$ are point-wisely continuous on $(0, 1)$. Thus, we can then state the following definition.

\begin{definition}\label{defwl}
Let $f,g\in \Gamma_0(H)$ and  $\lambda\in[0,1]$. We set
\begin{gather}\label{wLog1}
\mathfrak{L}_\lambda(f,g)= \int_{0}^{1} f\sharp_tg \,d\nu_\lambda(t),
\end{gather}
\begin{gather}\label{wLog2}
\mathbb{L}_\lambda(f,g)= \left(\int_{0}^{1} (f\nabla_tg)^*\,d\nu_\lambda(t)\right)^*.
\end{gather}
\end{definition}

\begin{remark}
When $\lambda=1/2$ then $\mathfrak{L}_{1/2}(f,g)=L(f,g)$ and $\mathbb{L}_{1/2}(f,g)=\mathbf{L}(f,g)$, where $L(f,g)$ and $\mathbf{L}(f,g)$ are defined by \eqref{logmean1}. For this reason, $\mathfrak{L}_\lambda(f,g)$ and $\mathbb{L}_\lambda(f,g)$ are both called here $\lambda$-weighted logarithmic functional mean of $f$ and $g$.
\end{remark}

Before stating some properties of $\mathfrak{L}_\lambda(f,g)$ and $\mathbb{L}_\lambda(f,g)$ we will analyze the novelty of these $\lambda$-weighted logarithmic functional means. First, as far as we know, $\mathfrak{L}_\lambda(f,g)$ and $\mathbb{L}_\lambda(f,g)$ appears to be new in their functional version. A question arises then from this latter information: do these functional means coincide in the scalar case? In fact, in the mono-dimensional case i.e when $H={\mathbb R}$, and choosing $f(x)=(1/2)ax^2$, $g(x)=(1/2)bx^2$ for some $a,b>0$, $\mathfrak{L}_\lambda(f,g)$ and $\mathbb{L}_\lambda(f,g)$ generate two $\lambda$-weighted logarithmic scalar means, respectively, given by
\begin{equation}\label{wlm1}
\mathfrak{L}_\lambda(a,b)= \int_{0}^{1} a\sharp_tb \,d\nu_\lambda(t),\;\; \mathbb{L}_\lambda(a,b)= \left(\int_{0}^{1} (a\nabla_tb)^{-1}\,d\nu_\lambda(t)\right)^{-1},
\end{equation}
where $a\sharp_tb:=a^{1-t}b^t$ is the weighted geometric mean and $a\nabla_tb:=(1-t)a+tb$ is the weighted arithmetic mean, of $a$ and $b$. Otherwise, another weighted logarithmic mean in the scalar version has been introduced in \cite{Pal et al(2016)} as follows
\begin{equation}\label{wlm2}
L_\lambda(a,b)=\dfrac{1}{\log a-\log b}\left(\dfrac{1-\lambda}{\lambda}\big(a-a^{1-\lambda}b^\lambda\big)+\dfrac{\lambda}{1-\lambda}\big(a^{1-\lambda}b^\lambda-b\big)\right),\; a\neq b.
\end{equation}
The functional version of $L_\lambda(a,b)$ has been investigated in \cite{RaiFur}.

Now, we are allowed to discuss the answer to the following question: are $\mathfrak{L}_\lambda(a,b)$, $\mathbb{L}_\lambda(a,b)$ and $L_\lambda(a,b)$, defined respectively by \eqref{wlm1} and \eqref{wlm2}, equal or different? First, we mention that if $\lambda=1/2$ then these three logarithmic means coincide and, they are all reduced to the standard logarithmic mean, namely
$$\mathfrak{L}_{1/2}(a,b)=\mathbb{L}_{1/2}(a,b)=L_{1/2}(a,b)=L(a,b):=\frac{b-a}{\log b-\log a},\;\; L(a,a)=a.$$
In the general case, they are mutually different as confirmed by the following example.

\begin{example}
Fixing $\lambda = 2/3$ and choosing various values for $a$ and $b$, we get, by performing calculations with Matlab software, the approximate values shown in Table \ref{table1} for the three logarithmic means.
 
\begin{table}[!h]
  \begin{tabular}{|c|c|c|c|c|}
  \hline
   a & b & $\mathfrak{L}_{2/3}(a,b)$ & $\mathbb{L}_{2/3}(a,b)$ & $L_{2/3}(a,b)$ \\ \hline
  2 & 4 & 3.232096013 & 3.225535716 & 3.228458409 \\ \hline
  1/2 & 3 & 1.843948110 & 1.827874186& 1.827005588 \\ \hline
  1/4 & 1/2 & 0.404012001 & 0.403191964 & 0.403557301 \\ \hline
  2 & 13 & 7.853396133 & 7.780949148 & 7.773936011 \\ \hline
 \end{tabular}
 \\
 \caption{Some numerical values for logarithmic means}\label{table1}
 \end{table}
 \end{example}
Now, we will state some more results about the previous weighted means.
\begin{proposition}\label{Estimate}
For any $f, g\in\Gamma_0(H)$ and $\lambda\in[0,1]$, we have
\begin{gather}\label{44f}
f!_\lambda g \leq \mathbb{L}_\lambda(f,g)\leq f\nabla_\lambda g,
\end{gather}
\begin{gather}\label{45f}
f!_\lambda\,g\le \big(\mathbb{L}_\lambda(f^*,g^*)\big)^* \leq  \mathfrak{L}_\lambda(f,g)\leq f\nabla_\lambda g.
\end{gather}
\end{proposition}
\begin{proof}
By applying \eqref{21} with the pointwise convex map  $f \mapsto f^*$, we get
\begin{equation}\label{x}
\big(f\nabla_\lambda g\big)^* \leq \int_0^1\big(f\nabla_t g\big)^*d\nu_\lambda(t)\leq f^*\nabla_\lambda g^*.
\end{equation}
If we recall that the map $f \mapsto f^*$ is pointwisely decreasing and satisfies $f^{**}\leq f$ for any $f$, we get \eqref{44f} after taking the conjugate of each term in \eqref{x}. Now, employing \eqref{IAGH} we get for any $t\in[0,1]$, $f!_t\,g\le f\sharp_t\,f\le f\nabla_t\,g$.
Multiplying all sides of these latter inequalities by $d\nu_\lambda(t)$ and then integrating with respect to $t \in [0, 1],$ we obtain
$$\Big(\mathbb{L}_\lambda(f^{*},g^{*})\Big)^*\leq \mathfrak{L}_\lambda(f,g)\leq \int_0^1 f\nabla_t\,g\;d\nu_\lambda(t).$$
Using the definition of $f\nabla_t\,g$ with \eqref{22} and then an elementary real-integration, we get the right inequality in \eqref{45f}.
The left inequality in \eqref{45f} is immediate from the definitions of the involved functional means with the help of some arguments as previous.
\end{proof}

\begin{proposition}\label{th2fenchel}
For any $a,\lambda\in[0,1]$ and $f, g\in \Gamma_0(H)$, the following inequalities hold
\begin{equation}\label{2.10f}
m(a,\lambda)\Big(f^*\nabla_a g^*-\big(f\nabla_a g\big)^*\Big)
\leq f^*\nabla_\lambda g^*-\Big({\mathbb L}_\lambda(f,g)\Big)^* 
\leq M(a,\lambda)\Big(f^*\nabla_ag^*-\big(f\nabla_ag\big)^*\Big),
\end{equation}
where $m(a,\lambda)$ and $M(a,\lambda)$ are the real numbers defined in Theorem \ref{th2}.
\end{proposition}
\begin{proof}
Applying \eqref{2.10} with the pointwise convex map $\Phi(f)=f^*$, with the definition of ${\mathbb L}_\lambda(f,g)$ and the fact that $f^{**}=f$ for any $f\in\Gamma_0(H)$, we get the desired inequalities.
\end{proof}

Taking $a=1/2$ in \eqref{2.10f}, we immediately obtain the following corollary.

\begin{corollary}\label{cor2f}
Let $\Phi:\mathcal{C}\rightarrow \widetilde{\mathbb{R}}^H$ be pointwise convex, $\lambda\in[0,1]$ and $f, g\in \mathcal{C}$, we have
\begin{gather}\label{3.11f}
\alpha(\lambda)\Big(f^*\nabla g^*-\big(f\nabla g\big)^*\Big)
\leq f^*\nabla_\lambda g^*-\Big({\mathbb L}_\lambda(f,g)\Big)^*
\leq \beta(\lambda)\Big(f^*\nabla g^*-\big(f\nabla g\big)^*\Big),
\end{gather}
where $\alpha(\lambda)$ and $\beta(\lambda)$ are the real numbers defined in Corollary \ref{cor1f}.
\end{corollary}

\end{document}